\documentclass{amsart}
\usepackage{amsmath,amscd,amsthm,amsfonts,amssymb,times}
\newtheorem{theorem}{Theorem}
\newtheorem{lemma}{Lemma}

\newtheorem{proposition}{Proposition}

\begin{document}

\title[Divisibility of class numbers of imaginary quadratic function fields]{Divisibility of class numbers of imaginary quadratic function fields by a fixed odd number}
\author{Pradipto Banerjee and Srinivas Kotyada}
\address{Institute of Mathematical Sciences,
CIT Campus, Tharamani, Chennai 600 113, India}
\email[Pradipto Banerjee]{pradipto@imsc.res.in}
\email[Srinivas Kotyada]{srini@imsc.res.in}
\keywords{Divisibility, Class numbers, Quadratic extensions, Function fields}
\subjclass[2000]{11R29 (primary); 11R11 11R58 (secondary)}
\maketitle

\begin{abstract}
In this paper we find a new lower bound on the number of imaginary quadratic extensions of the function field $\mathbb{F}_{q}(x)$ whose class groups have elements
of a fixed odd order. More precisely, for $q$, a power of an odd prime, and $g$ a fixed odd positive integer $\ge 3$, we show that for every $\epsilon >0$, there
are $\gg q^{L(\frac{1}{2}+\frac{3}{2(g+1)}-\epsilon)}$ polynomials $f \in \mathbb{F}_{q}[x]$ with $\deg f=L$, for which the class group of the quadratic extension
$\mathbb{F}_{q}(x, \sqrt{f})$ has an element of order $g$. This sharpens the previous lower bound $q^{L(\frac{1}{2}+\frac{1}{g})}$ of Ram Murty. Our result is a
function field analogue to a similar result of Soundararajan for number fields. 
\end{abstract}

\section{Introduction}\label{intro}
For a square-free integer $D$, let Cl$(-D)$ denote the ideal class group of $\mathbb{Q}(\sqrt{-D})$, and let $h(-D)=$ \#Cl$(-D)$ denote the class number.
In his 1801 Disquisitiones Arithmeticae, Gauss put forward the problem of finding all positive square-free $D$ such that $h(-D)$ is some fixed number $C$.
Heegner \cite{Heeg}, Baker \cite{Baker1} and Stark \cite{Stark1} solved Gauss's problem completely for $C=1$. Subsequently, Baker \cite{Baker2} and Stark 
\cite{Stark2} provided solutions to the case $C=2$. Recently, Watkins \cite{Wat} extended the range of the complete solutions to Gauss's problem for $C \le 100$.
\vskip 5pt

A related problem of interest is to determine the existence of $g$-torsion subgroups of Cl$(-D)$ for positive integers $g$. Gauss studied the case $g=2$. Davenport
and Heilbronn \cite{DH} proved that the proportion of $D$ with $3 \nmid h(-D)$ is at least $1/2$. For any $g$ the infinitude of such fields was established by Nagell \cite{N}, 
Honda \cite{H}, Ankeny and Chowla \cite{AC}, Hartung \cite{Ht}, Yamamoto \cite{Y} and Weinberger \cite{W}.
\vskip 5pt

For a positive integer $g$, let $N_{g}(X)$ denote the number of positive square-free $D \le X$ such that $g|h(-D)$. Gauss's genus theory (for reference see \cite{BS})
demonstrates that $2|h(-D)$ whenever $D$ is a product of at least two odd prime numbers. This in particular implies that $N_{2}(X)\sim 6X/\pi^{2}$. In general it is 
believed that $N_{g}(X)\sim C_{g}X$ for some positive constant $C_{g}$. For odd primes $g$, Cohen and Lenstra \cite{CL} conjectured that
\[
C_{g}=\frac{6}{\pi^{2}}\Big(1-\prod_{i=1}^{\infty}\Big(1-\frac{1}{g^{i}}\Big)\Big).
\]
Ankeny and Chowla \cite{AC} were among the first to achieve an estimate for $N_{g}(X)$ for $g\ge 3$. Although they did not explicitly point this out, their method shows 
that for $g \ge 3$, $N_{g}(X)\gg X^{1/2}$. Recently, Murty \cite{M} improved this lower bound to $N_{g}(X) \gg X^{\frac{1}{2}+\frac{1}{g}}$, which was subsequently sharpened by 
Soundararajan \cite{S} who showed
\begin{align*}
N_{g}(X)\gg
\begin{cases}
X^{\frac{1}{2}+\frac{2}{g}-\epsilon} \quad & \text{if}\quad g\equiv 0 \pmod{4} \\
X^{\frac{1}{2}+\frac{3}{g+2}-\epsilon} \quad & \text{if}\quad g\equiv 2 \pmod{4}.
\end{cases}
\end{align*}
\vskip 5pt

For $q$, a power of an odd prime, we define $k :=\mathbb{F}_{q}(x)$ to be the function field over the finite field $\mathbb{F}_{q}$ and $\mathcal{A} :=\mathbb{F}_{q}[x]$, 
its ring of integers. For a square-free $f\in \mathcal A$, we will denote the quadratic field extension $k(\sqrt{f})$ by $K$, and its ring of integers $\mathcal{A}[\sqrt{f}]$ by 
$\mathcal{B}$. The function field analogue of the class number divisibility problem was initiated by Emil Artin \cite{Artin}. Friesen \cite{F} 
constructed infinitely many polynomials $f\in \mathcal{A}$ of even degree such that the class groups for $K$ have an element of order $g$ where $g$ is not divisible by $q$. 
Friedman and Washington \cite{FW} have studied the Cohen-Lenstra conjecture in the function field case. In \cite{MC}, Murty and Cardon proved that for $q\ge 5$ there are 
$\gg q^{L(\frac{1}{2}+\frac{1}{g})}$ polynomials $f\in \mathcal{A}$ with $\deg(f)\le L$ such that the class groups for the quadratic extensions $K$ have an element of order $g$, 
which is analogous to the result $N_{g}(X) \gg X^{\frac{1}{2}+\frac{1}{g}}$ of Murty \cite{M}. In \cite{CM}, Chakraborty and Mukhopadhyay have shown that there are $\gg q^{L/2g}$ 
monic polynomials $f \in \mathcal{A}$ of even degree with $\deg(f)\le L$ such that the ideal class group of the (real) quadratic extensions $K$ have an element of order $g$. 
This is a function field analogue of Murty's result \cite{M} $N_{g}(X) \gg X^{1/2g}$ for real quadratic number fields. 
\vskip 5pt

The case when $\deg f$ is odd is analogous to the case of an imaginary quadratic number field in which the prime at infinity ramifies and the unit group has rank $0$. 
Recently, Merberg \cite{Mer} used a function field analogue to the Diophantine method of Soundararajan \cite{S} for finding  imaginary quadratic function fields whose class groups have 
elements of a given order. He further proved that there are infinitely many such fields whose class numbers are not divisible by any odd prime distinct from the characteristic.
\vskip 5pt

In the present work, we sharpen the lower bound of Murty and Cardon for imaginary quadratic extensions of $k$, and for odd $g \ge 3$. Specifically, we prove the following
\begin{theorem}\label{thm1}
Let $g \ge 3$ be a fixed positive odd integer. Let $q$ be a power of an odd prime. For odd $L$, let $N_{g}(L)$ denote the number of square-free polynomials 
$f \in \mathbb{F}_{q}[x]$ with $\deg f \le L$ such that the class group of the quadratic extension $\mathbb{F}_{q}(x,\sqrt{f})$ contain an element of order $g$. Then, for sufficiently 
large $L$ we have
\[
N_{g}(L) \gg q^{L(\frac{1}{2}+\frac{3}{2(g+1)}-\epsilon)}.
\] 
\end{theorem} 

We will work with polynomials $f$ with $\deg f =L$. This, however we note that does not affect the statement of our result. We will use ideas from \cite{S} to achieve our result. 
From our construction of the quadratic extensions of $\mathbb{F}_{q}(x)$ it will become evident that the case when $g \equiv 0 \pmod{4}$ cannot be handled by our method. However, 
we remark that by a straightforward group theoretic argument and Theorem \ref{thm1}, a new lower bound when $g \equiv 2 \pmod{4}$ can be achieved if one can first settle the function 
field analogue of Gauss's genus theory.
\vskip 5pt

For basic function field related concepts, we refer the reader to \cite{R}. We will denote by $\mathbb{F}_{q}^{\times}$ the multiplicative group of non-zero elements in $\mathbb{F}_{q}$. 
For an integer $U$, we let $\pi(U)$ count the number of irreducible monic polynomials of degree $U$. For a $f \in \mathcal{A}$, define the norm $|f|$ of $f$ as $|f|:=q^{\deg f}$, and let 
sgn$(f)$ denote the leading coefficient of $f$. Let the M\"{o}bius function $\mu(f)$ be $0$ if $f$ is not square-free, and $(-1)^{t}$ if $f$ is a constant times a product of $t$ distinct 
irreducible monic polynomials in $\mathcal{A}$. We will let $d(f)$ denote the number of distinct monic divisors of $f$ (including $f/\text{sgn}f$). We further define the Euler function 
$\phi(f)$ to be the order of the unit group $(\mathcal{A}/f\mathcal{A})^{\times}$ of the ring $\mathcal{A}/f\mathcal{A}$. It can be verified that
\[
\phi(f)=|f|\prod_{p|f}\big(1-\frac{1}{|p|}\big),
\]
where the product is taken over irreducible monic polynomials. For $a$, $b$ in $\mathcal{A}$, the symbol $(a,b)$ will denote the greatest common monic divisor of $a$ and $b$, and 
$\Big(\frac{a}{b}\Big)$ denotes the Jacobi symbol whenever relevant. We will let. For functions $F$ and $G$, we will use the notation $F \asymp G$ 
whenever $F \gg \ll G$. Finally, we would like to point out to the reader that the `$\epsilon$'s appearing at different places are different.
\vskip 5pt

We prove our result by first giving a criteria for the existence of elements of order $g$ in Cl$(f)$, the class group of $K$. This will be achieved in Section \ref{divisibility}. In order to 
obtain the lower bound in the theorem, we need to count the number of square-free $f$ meeting the divisibility criteria. We will do this in Section \ref{count}. Sections \ref{lemproof1} 
and \ref{lemproof2} provide the technical details needed in Section \ref{count}. The last section contains the conclusion of the proof.

\section{A divisibility criteria for the class number of $\mathbb{F}_{q}(x, \sqrt{f})$}\label{divisibility}
Define the norm $N(a) \in \mathcal A$ of an element $a \in \mathcal B$ as $N(a)=a \bar{a}$, where $\bar{a}$ is the conjugate of $a$. For an ideal $\mathfrak{v}$ in $\mathcal B$, we consider 
the ideal $\mathfrak{u}$ in $\mathcal A$ generated by the set $\{N(a): a \in \mathfrak{v}\}$. Since $\mathcal A$ is a principal ideal domain, the ideal $\mathfrak{u}$ is principal, 
say $\mathfrak{u}=(b)$, where $b \in \mathcal A$. We define the norm $N(\mathfrak{v})$ of the ideal $\mathfrak{v}$ as $q^{\deg b}$. We note that for a principal ideal $(a)$ in 
$\mathcal B$, $N((a))=q^{\deg N(a)}$.
\vskip 5pt

In the following proposition, we construct quadratic extensions of $k$ whose class groups contain an element of order $g$.

\begin{proposition}\label{construction}
Let $g \ge 3$ be an odd positive integer. Let $f \in \mathcal A$ be a square-free polynomial of odd degree. If there exist nonzero $m$, $n$, $t \in \mathcal A$ such that $t^{2}f=n^{2}-m^{g}$ 
with $(m,n)=1$ and $\deg m^{g}>\max\{\deg n^{2}, \deg t^{4}\}$, then the class group for $K$ has an element of order $g$. 
\end{proposition}

\begin{proof}
Suppose $m$, $n$ and $t$ as in the lemma exist. Rewriting $t^{2}f=n^{2}-m^{g}$ as $m^{g}=n^{2}-t^{2}f$, we see that the ideal $(m)^{g}$ factors in $\mathcal B$ as
\[
(m)^{g}=(n+t\sqrt{f})(n-t\sqrt{f}).
\]
We note that any common divisor $\mathfrak{d}$ of the ideals $(n+t\sqrt{f})$ and $(n-t\sqrt{f})$ contains $2n$. As $2$ is a unit in $\mathcal A$, we deduce that $n \in \mathfrak{d}$. 
On the other hand $\mathfrak{d}$ also contains $m^{g}$, but $(m^{g},n)=1$. Thus $\mathfrak{d}=\mathcal B$, that is the ideals $(n+t\sqrt{f})$ and $(n-t\sqrt{f})$ are co-prime in $\mathcal B$.
\vskip 5pt
\noindent
Thus there exist ideals $\mathfrak{a}$ and $\mathfrak{a}'$ in $\mathcal B$ such that $(n+t\sqrt{f})=\mathfrak{a}^{g}$, and $(n-t\sqrt{f})=\mathfrak{a}'^{g}$. 
\vskip 5pt
We claim that the ideal class of $\mathfrak{a}$ has order $g$. Assume otherwise that there is a positive integer $r<g$ such that $\mathfrak{a}^{r}$ is principal, say 
$\mathfrak{a}^{r}=(u+v\sqrt{f})$ for some $u$, $v\in \mathcal A$. It is clear that $r|g$. Taking norm we have $N(\mathfrak{a})^{r}=q^{\deg (u^{2}-v^{2}f)}$. We also have 
$(n+t\sqrt{f})=(u+v\sqrt{f})^{g/r}$. Since $t \neq 0$, it immediately follows that $v \neq 0$. Thus $v^{2}f \neq 0$ has odd degree, and since $u^{2}$ has even degree, 
$\deg (u^{2}-v^{2}f)\ge \deg f$.
\vskip 5pt
\noindent
Therefore $N(\mathfrak{a})^{r}=q^{\deg (u^{2}-v^{2}f)} \ge q^{\deg f}$. On the other hand, 
\[
N(\mathfrak{a})^{g}=q^{\deg (n^{2}-t^{2}f)}=q^{\deg m^{g}}=q^{g\deg m}.
\]
Thus $N(\mathfrak{a})=q^{\deg m}$. 
\vskip 5pt
\noindent
Now from $q^{r \deg m}=N(\mathfrak{a})^{r} \ge q^{\deg f}$ we see that 
\begin{equation}\label{ineq1}
r\deg m \ge \deg f = \deg \big(\frac{n^{2}-m^{g}}{t^{2}}\big)=g\deg m -2\deg t.
\end{equation}
The last equality above follows from our assumption that $\deg m^{g}>\max\{\deg n^{2}, \deg t^{4}\}$.
\noindent
Rearranging terms in inequality (\ref{ineq1}), we have $\deg m \le \frac{2\deg t}{g-r}$. But from our assumption that $\deg m^{g}> \deg t^{4}$, it now follows that
\[
\frac{4\deg t}{g}< \deg m \le\frac{2\deg t}{g-r},
\]
giving rise to $\frac{g}{r}< 2$, and there by contradicting the fact that $r|g$ since $g \ge 3$. This proves our claim and hence the proposition.
\end{proof}

\section{Counting square-free $f$}\label{count}

In this section we shall obtain a lower bound on the number of square-free $f \in \mathcal A$ meeting the criteria of Proposition \ref{construction}. 
The bound obtained in this section will depend on some parameter $T$ to be determined in Section \ref{thmproof}(see (\ref{TLEQ})).
\vskip 5pt

Thus we will be interested in counting the number of square-free polynomials $f \in \mathcal A$ satisfying
\begin{equation}\label{dioeq1}
n^{2}-m^{g}=t^{2}f, \quad (m,n)=1 \quad \quad \text{and} \quad \deg m^{g}>\max\{n^{2}, t^{4}\}.
\end{equation}
Let $\deg m=M$, $\deg n=N$, $\deg t=T$ and $\deg f=L$. In view of Proposition \ref{construction} we assume that
\begin{equation}\label{MNT}
T< L/2,\quad Mg=2T+L \quad \quad \text{and} \quad N=T+\frac{L}{2}-1.
\end{equation}
\noindent
From the above choice of $M$, $N$ and $T$ it follows that 
\[
Mg >\max\{2N, 4T\},
\]
that is $\deg m^{g}>\max\{n^{2}, t^{4}\}$. Thus if $f$ admits a solution to the (\ref{dioeq1}), then
by Proposition \ref{construction}, Cl$(f)$ has an element of order $g$.
\vskip 5pt
Let $N_{g}(L,T)$ count the number of square-free $f$ with $\deg f=L$ and satisfying (\ref{dioeq1}). 
For a square-free polynomial $f \in \mathcal A$ of degree $L$, let $R(f)$ denote the number of solutions in monic $m$, $n$ and $t$ to 
(\ref{dioeq1}). If we define the characteristic function $\chi(f)$ as 

\[
\chi(f)=
\begin{cases}
0 \quad \text{if}\quad \mathcal{R}(f)=0 \\
1 \quad \text{if}\quad \mathcal{R}(f)\neq 0,
\end{cases}
\]
then we can write $N_{g}(L,T)$ as
\[
N_{g}(L,T)=\sum_{\deg f=L} \chi(f).
\]
By Cauchy-Schwarz inequality we have
\[
(\sum_{\deg f=L} \chi(f)^{2})(\sum_{\deg f=L}\mathcal{R}(f)^{2}) \ge (\sum_{\deg f=L}\chi(f)\mathcal{R}(f))^{2},
\]
which can be rewritten as
\begin{equation}\label{rd1}
N_{g}(L,T) \ge (\sum_{\deg f=L}\mathcal{R}(f))^{2}(\sum_{\deg f=L}\mathcal{R}(f)^{2})^{-1}.
\end{equation}
Thus, in order to determine a lower bound on $N_{g}(L,T)$, we need to establish a lower bound on $(\sum_{\deg f=L}\mathcal{R}(f))^{2}$ and an upper bound on $\sum_{\deg f=L}\mathcal{R}(f)^{2}$.
\vskip 5pt
In the next section we will obtain the lower bound on $(\sum_{\deg f=L}\mathcal{R}(f))^{2}$ by establishing the following lemma.
\begin{lemma}\label{lowerbound}
$\sum_{\deg f=L}\mathcal{R}(f) \asymp q^{M+N-T}$.
\end{lemma}
By a counting argument, we will show in Section \ref{lemproof2} that
\begin{lemma}\label{upperbound}
$\sum_{\deg f=L}\mathcal{R}(f)\big(\mathcal{R}(f)-1\big) \ll q^{\epsilon L +2M+2T}$.
\end{lemma}
Below we demonstrate how Lemma \ref{lowerbound} and Lemma \ref{upperbound} give a lower bound on $N_{g}(L,T)$.
\vskip 5pt
\noindent
Observe that
\[
\sum_{\deg f=L}\mathcal{R}(f)^{2}=\sum_{\deg f=L}\mathcal{R}(f)\big(\mathcal{R}(f)-1\big)+\sum_{\deg f=L}\mathcal{R}(f) \ll q^{M+N-T}+ q^{\epsilon L +2M+2T}.
\]
In order to achieve an upper bound on $\sum_{\deg f=L}\mathcal{R}(f)^{2}$, we will optimally choose the parameter $T$ so that
\begin{equation}\label{MNT1}
M+N-T \le \epsilon L +2M +2T.
\end{equation}
Thus
\begin{equation}\label{rd2}
\sum_{\deg f=L}\mathcal{R}(f)^{2} \ll q^{\epsilon L +2M+2T}.
\end{equation}
\noindent
Therefore from (\ref{rd1}), (\ref{rd2}) and Lemma \ref{lowerbound} we have
\[
N_{g}(L,T) \gg \frac{q^{2(M+N-T)}}{q^{\epsilon L +2M+2T}} = q^{2N-4T- \epsilon L}.
\]
\noindent
Putting the value of $N$ from (\ref{MNT}) we get

\begin{equation}\label{nglt}
N_{g}(L,T) \gg q^{L-2T-2-\epsilon L} \gg q^{L-2T-\epsilon L}.
\end{equation}
\noindent
The lower bound in Theorem \ref{thm1} will be achieved by suitably choosing the parameter $T$ in Section \ref{thmproof}.

\section{Proof of Lemma \ref{lowerbound}}\label{lemproof1}
Let $(m,n,t) \in \mathcal{A}^{3}$ be a tuple of pairwise relatively prime monic polynomials with $\deg m=M$, $\deg n=N$ and $\deg t=T$, where $M$, $N$ and $T$ satisfy (\ref{MNT}), 
and satisfying $n^{2}\equiv m^{g} \pmod{t^{2}}$. We define sets $\mathcal{S}_{1}$, $\mathcal{S}_{2}$ and $\mathcal{S}_{3}$ of such tuples $(m,n,t) \in \mathcal{A}^{3}$ as follows.

\begin{align*}
& \mathcal{S}_{1}=\{(m,n,t): p^{2}\nmid \frac{n^{2}-m^{g}}{t^{2}} \text{ for all monic primes } p \text{ with } \deg p \le \log L \}, \\
& \mathcal{S}_{2}=\{(m,n,t): p^{2}| \frac{n^{2}-m^{g}}{t^{2}} \text{ for some monic primes } p \text{ with } \log L < \deg p \le Q \}\quad \text{and} \quad \\
& \mathcal{S}_{3}=\{(m,n,t): p^{2}| \frac{n^{2}-m^{g}}{t^{2}} \text{ for some monic primes } p \text{ with } Q < \deg p \}.
\end{align*}
\noindent
Here logarithms are taken to the base $q$, and $Q$ is some real parameter to be described below.
\vskip 5pt

Let $N_{i}=|\mathcal{S}_{i}|$ for $i=1,2,3$. The sum we desire is $N_{1}+O(N_{1}+N_{2})$. We shall show below that by choosing $Q:= (L-T+2 \log L)/3$, one obtains

\begin{align*}
& N_{1} \asymp q^{M+N-T}+o(q^{M+\frac{L}{3}+\frac{2T}{3}}),\\
& N_{2} \ll q^{M+N-T}/L+o(q^{M+\frac{L}{3}+\frac{2T}{3}}) \quad \text{and}\quad \\
& N_{3} =o(q^{M+\frac{L}{3}+\frac{2T}{3}}).
\end{align*}
\noindent
Observe that for $L >4T$, it follows from (\ref{MNT}) that $M+N-T \ge M+(L/3)+(2T/3)$, and hence $N_{1} \asymp q^{M+N-T}$, and $N_{2}$, $N_{3}$ are small. 
The choice of $T$ in (\ref{TLEQ}), Section \ref{thmproof}
guarantees that $L > 4T$. Thus it follows that
\[
\sum_{\deg f =L}R(f) \asymp q^{M+N-T}.
\]

\noindent
\textit{Estimation of} $N_{1}$: For fixed monic $m$ and $t$ with $\deg m=M$ and $\deg t=T$, we count the number of monic polynomials $n$ with $\deg n=N$ such that 
$n^{2}\equiv m^{g} \pmod{t^{2}}$, and $p^{2}$ does not divide $\frac{n^{2}-m^{g}}{t^{2}}$ for all irreducible monic $p$ with $\deg p \le \log L$.
\vskip 5pt

Let $\rho_{m}(l)$ denote the number of solutions $\pmod{l}$ to the congruence $n^{2}\equiv m^{g} \pmod{l}$. It can be verified (for example see \cite{MC} 
or \cite{M}) that if $p \nmid m$ is irreducible, then for $\alpha \ge 1$,
\begin{equation}\label{lb2}
\rho_{m}(p^{\alpha})=\rho_{m}(p)= 1+\Big(\frac{m^{g}}{p}\Big)=1+\Big(\frac{m}{p}\Big),
\end{equation}
as $g$ is odd.
\vskip 5pt

Set $P=\prod_{\deg p \le \log L}p$, where the product is taken over all irreducible monic polynomials $p$ so that $\sum_{l^{2}|(f,P^{2})}\mu(l)=1$ or $0$ 
depending on whether $p^{2}\nmid f$ for all $p$ with $\deg p \le \log L$ or not. Here $l$ is assumed to be monic. Thus in order to estimate $N_{1}$, the sum over $n$ we seek is

\begin{equation}\label{lb3}
\sum_{\substack{\deg n=N\\ n^{2}\equiv m^{g} \pmod{t^{2}}\\ (n,m)=1}}\sum_{l^{2}|\big(\frac{n^{2}-m^{g}}{t^{2}},P^{2}\big)}\mu(l)
=\sum_{\substack{l|P\\ (l,m)=1}}\mu(l)\sum_{\substack{\deg n=N\\ n^{2}\equiv m^{g} \pmod{l^{2}t^{2}}}}1.
\end{equation}
\noindent
If $N \ge \deg l^{2}t^{2}$ then
\[
\sum_{\substack{\deg n=N\\ n^{2}\equiv m^{g} \pmod{l^{2}t^{2}}}}1=\frac{|n|}{|l^{2}t^{2}|} \rho_{m}(l^{2}t^{2})=\frac{q^{N-2T}\rho_{m}(l^{2}t^{2})}{|l^{2}|},
\]
\noindent
while if $N \le \deg l^{2}t^{2}$ then
\[
\sum_{\substack{\deg n=N\\ n^{2}\equiv m^{g} \pmod{l^{2}t^{2}}}}1 \le \rho_{m}(l^{2}t^{2}).
\]
\noindent
Thus the sum in (\ref{lb3}) is

\begin{eqnarray*}
&=& \sum_{\substack{l|P\\ (l,m)=1}}\mu(l)\frac{|n|}{|l^{2}t^{2}|}\rho_{m}(l^{2}t^{2}) +O\Big(\sum_{\substack{l|P\\ (l,m)=1}}\rho_{m}(l^{2}t^{2})\Big) \nonumber \\
&=& q^{N-2T}\rho_{m}(t^{2})\sum_{\substack{l|P\\ (l,m)=1}}\frac{\mu(l)}{|l|^{2}}\rho_{m}\Big(l/(l,t)\Big)+O\Big(\sum_{\substack{l|P\\ (l,m)=1}}\rho_{m}(l^{2}t^{2})\Big), \nonumber \\
\end{eqnarray*}
\noindent
which can be written as 
\begin{equation}\label{lb4}
 q^{N-2T}\rho_{m}(t^{2})\prod_{\substack{p|P\\ p-\text{monic}\\ (p,m)=1}}\Big(1-\frac{\rho_{m}\big(p/(p,t)\big)}{|p|^{2}}\Big)+O\Big(\sum_{\substack{l|P\\ (l,m)=1}}\rho_{m}(l^{2}t^{2})\Big),
\end{equation}
where the product is taken over irreducible monic polynomials $p$.
\vskip 5pt
\noindent
It can be easily seen from $\rho_{m}\big(p/(p,t)\big)=1+ \Big(\frac{m}{p}\Big) \le 2$ that 
\[
\prod_{\substack{p|P\\ p-\text{monic}\\ (p,m)=1}}\Big(1-\frac{\rho_{m}\big(p/(p,t)\big)}{|p|^{2}}\Big) \asymp 1. 
\]
Therefore the main term in (\ref{lb4}) is $\asymp q^{N-2T}\rho_{m}(t^{2})$.
\vskip 5pt
\noindent
For the error term in (\ref{lb4}), we first note from (\ref{lb2}) that
\[
\rho_{m}(l^{2}t^{2})=\rho_{m}(lt)=\prod_{p|lt}\rho_{m}(p)=\prod_{p|lt}\Big(1+\Big(\frac{m}{p}\Big)\Big)\le \prod_{p|lt}2 \le d(lt).
\]
As $l^{2}t^{2}$ divides $n^{2}-m^{g}$, we have from (\ref{MNT}) that
\[
2 \deg l +2 \deg t \le Mg = L +2T= L + 2 \deg t.
\]
Therefore $\deg l \le L/2$. Also from (\ref{MNT}) we have $\deg t = T < L/2$. Hence $\deg lt \le L$. 
\vskip 5 pt
\noindent
It can be verified that for polynomials $r(x) \in \mathcal A$ with $\deg r \le X$, $d(r)=O(q^{\epsilon X})$. Therefore we conclude that
\[
\rho_{m}(l^{2}t^{2}) \le d(lt) =O(q^{\epsilon L}).
\]
Thus the error term in (\ref{lb4}) is $O(d(P) q^{\epsilon L})$. We shall obtain an upper bound for $d(P)$ below.
\vskip 5pt
\noindent
Clearly, we have 

\begin{equation}\label{lb5}
d(P)= 2^{\pi(1)+\pi(2)\cdots +\pi(\log L)}.
\end{equation}


The following lemma gives us an upper bound for $\pi(U)$ for $U \in \mathbb N$.

\begin{lemma}\label{lemmasec41}
For $U \in \mathbb N$, $\pi(U)\le q^{U}/U$.
\end{lemma}

\begin{proof}
Since $q^{U} =\sum_{D|U}D \pi(D)$, we have in particular, for $D=U$ that
\[
U\pi(U) \le \sum_{D|U}D \pi(D)=q^{U},
\]
and hence the lemma.
\end{proof}
\noindent
Therefore from (\ref{lb5}) we have
\[
d(P)=2^{\pi(1)+\pi(2)\cdots +\pi(\log L)}\le 2^{q+q^{2}/2\cdots +q^{\log L/L}}< qL.
\]
\noindent
Thus the error term in (\ref{lb4}) is $O(q^{\epsilon L})$.
\vskip 5pt
\noindent
Therefore the sum in (\ref{lb3}) is 
\[
\asymp q^{N-2T}\rho_{m}(t^{2})+O(q^{\epsilon L}). 
\]
Now, summing over all monic $m$ with $\deg m=M$, and monic $t$ with $\deg t=T$ we have 
\begin{equation}\label{lb6}
N_{1}\asymp q^{M+N-T}\sum_{\substack{\deg m=M \\ \deg t=T}}\rho_{m}(t^{2})+ O\big(q^{\epsilon L+M+T}\big).
\end{equation}
We now show that the error term in (\ref{lb6}) is $o(q^{M+\frac{L}{3}+\frac{2T}{3}})$.
We choose $0< \delta < \frac{1}{2}$ so that $q^{L/2}=o(q^{L(1-\delta)})$. Since from (\ref{MNT}) we have $T < L/2$, hence $q^{T} < q^{L/2}=o(q^{L(1-\delta)})$.
\vskip 5pt
\noindent
Taking $\epsilon =\frac{\delta}{3}$, we have $q^{T/3}=o(q^{L/3}q^{\epsilon L})$, that is $q^{\epsilon L}=o(q^{L/3}q^{-T/3})$. 
\vskip 5pt
\noindent
Thus from (\ref{lb6}) we have
\begin{equation}\label{lb61}
N_{1} \asymp q^{N-2T}\sum_{\substack{\deg m=M \\ \deg t=T}}\rho_{m}(t^{2})+o(q^{M+\frac{L}{3}+\frac{2T}{3}}).
\end{equation}
We next show that 
\[
\sum_{\substack{\deg m=M \\ \deg t=T}}\rho_{m}(t^{2}) \asymp q^{M+T}.
\]
 \vskip 5pt
\noindent
In order to prove this result we will need a couple of lemmas.

\begin{lemma}\label{lemmasec42}
For an integer $U \ge 2$, we have
\[
\sum_{\substack{y-\text{monic} \\ \deg y=U}}\mu(y)=0.
\]
\end{lemma}

\begin{proof}
For $j \ge 0$, let
\[
H(j)=\sum_{\substack{y-\text{monic} \\ \deg y=j}}\mu(y)
\]
Then it follows that the Dirichlet series
\begin{equation}\label{lb7}
\sum_{y-\text{monic}}\frac{\mu(y)}{|y|^{s}}=\sum_{j=0}^{\infty}\frac{H(j)}{q^{js}}.
\end{equation}
On the other hand we have from the definition of the \textit{zeta} function \cite{R} in $\mathcal A$ that
\[
\sum_{y-\text{monic}}\frac{\mu(y)}{|y|^{s}}=\zeta_{\mathcal A}(s)^{-1}=1-q^{1-s}.
\]
Thus, using the substitution $u=q^{-s}$ in (\ref{lb7}) we have
\[
\sum_{j=0}^{\infty}H(j)u^{j}=1-qu.
\]
Comparing the coefficients of $u^{j}$ on both sides we have the result of our lemma.
\end{proof}

The next lemma is based upon Lemma 17.10, Proposition 17.11 and Proposition 17.12 of \cite{R} which we state without proof as follows.

\begin{lemma}\label{lemmasec43}
Suppose $b \notin \mathbb{F}_{q}^{\times}$ is not a square in $\mathcal A$, and let $\deg b =B$. Then
\vskip 5pt

(i) for $D \ge B$,
\[
\sum_{\substack{a-\text{monic} \\ \deg a =D}}\Big(\frac{b}{a}\Big)=0.
\]
\vskip 5pt
(ii) For $1 \le D \le B -1$,
\[
\sum_{\substack{b-\text{monic} \\ \deg b =B}}\sum_{\substack{a-\text{monic} \\ \deg a =D}}\Big(\frac{b}{a}\Big)=(q-1)\Phi(D/2, M),
\]
where 

\begin{align*}
\Phi(D/2, M)=
\begin{cases}
\Big(1-\frac{1}{q}\Big)q^{M+D/2} \quad & \text{if} \quad D \equiv 0 \pmod{2} \\
\quad 0 \quad  \quad & \text{if} \quad D \equiv 1 \pmod{2}.
\end{cases}
\end{align*}

\end{lemma}

We are now ready to estimate the average value of $\rho_{m}(t^{2})$.

\begin{lemma}\label{average}
Assume that $m$ and $t \in A$ are monic and relatively prime. Then we have
\[
\sum_{\deg m=M}\sum_{\deg t=T}\rho_{m}(t^{2})\asymp q^{M+T} +O(q^{M/2+T}) \asymp q^{M+T}.
\]
\end{lemma}

\begin{proof}
We have 
\[
\rho_{m}(t^{2})=\rho_{m}(t)=\prod_{p|t}\Big(1+\Big(\frac{m}{p}\Big)\Big)=\sum_{d|t}\mu^{2}(d)\Big(\frac{m}{d}\Big).
\]

We derive our result by showing that the main contribution in the above sum comes from $d=1$. For $d=1$, the sum over $t$ we are interested in is
\begin{align*}
\sum_{\substack{\deg t=T\\ (t,m)=1}}1 &= \sum_{\substack{\deg t=T\\ s|t}}\sum_{s|m}\mu(s)=\sum_{s|m}\mu(s)\sum_{\substack{\deg t=T\\ s|t}}1\\
&=\sum_{s|m}\mu(s)\sum_{\substack{l\\ ls=t}}1=\sum_{s|m}\mu(s)\sum_{\substack{l\\ \deg l=T-\deg s}}1\\
&=\sum_{s|m}\mu(s)q^{T-\deg s}=q^{T}\prod_{p|m}\Big(1-\frac{1}{q^{\deg p}}\Big)\\
&=q^{T}\frac{\phi(m)}{|m|}=q^{T-M}\phi(m).
\end{align*}
Now summing over $m$, and using Proposition 2.7 of \cite{R} we have
\[
q^{T-M}\sum_{\deg m=M}\phi(m)=q^{T-M}\cdot q^{2M}\big(1-\frac{1}{q}\big).
\]
Thus the contribution from $d=1$ is indeed $\asymp q^{M+T}$.
\vskip 5pt
We next demonstrate that the contribution from $d \neq 1$ is $O(q^{M/2+T})$. The sum we seek to bound is
\[
\sum_{\deg m=M}\sum_{\substack{\deg t=T\\ (t,m)=1}}\sum_{\substack{d|t\\ d \neq 1}} \mu^{2}(d)\Big(\frac{m}{d}\Big).
\]
Let us denote $\deg d$ by $Z$. We split the above sum into $1 \le Z \le M$, and $Z \ge M+1$, where $M=\deg m$. The first sum (after changing the order of summation) is
\[
\sum_{\substack{\deg t=T\\ (t,m)=1}}\sum_{\substack{d|t\\ Z \le M}}\mu^{2}(d)\sum_{\deg m=M}\Big(\frac{m}{d}\Big).
\]
\noindent
Observe that if $d$ is a square then $\mu^{2}(d)=0$, and if $d$ is not a square, then from quadratic reciprocity law we have
\[
\Big(\frac{m}{d}\Big)\Big(\frac{d}{m}\Big)=(-1)^{\frac{q-1}{2}(\deg m)(\deg d)}\text{sgn}(m)^{\deg d}=(-1)^{\frac{q-1}{2}MZ}.
\] 
Since $d \neq 1$, Lemma \ref{lemmasec43} implies
\[
\sum_{\deg m =M}\Big(\frac{m}{d}\Big)=(-1)^{\frac{q-1}{2}MZ}\sum_{\deg m =M}\Big(\frac{d}{m}\Big)=0
\]
for $\deg d = Z \le M$. So the first sum is $0$. 
\vskip 5pt

We now consider the second sum:
\begin{align*}
\sum_{\deg m=M}\sum_{\substack{\deg t=T\\ (t,m)=1}}\sum_{\substack{d|t\\M+1 \le Z \le T}}\mu^{2}(d)\Big(\frac{m}{d}\Big)
&=\sum_{\deg m=M}\sum_{M+1 \le Z \le T}\sum_{\substack{\deg d=Z\\ (d,m)=1}}\mu^{2}(d)\Big(\frac{m}{d}\Big)q^{T-Z}\\
&=q^{T}\sum_{M+1 \le Z \le T}q^{-Z}\sum_{\deg m=M}\sum_{\substack{\deg d=Z\\ (d,m)=1}}\mu^{2}(d)\Big(\frac{m}{d}\Big).
\end{align*}
\noindent
Since $\big(\frac{m}{d}\big)=0$ when $(d,m) \neq 1$, we can ignore the condition $(d,m)=1$ in the above summation. Let us denote the inner sum above by
\[
S:=\sum_{\deg m=M}\sum_{\deg d=Z}\mu^{2}(d)\big(\frac{m}{d}\big).
\]
We write $d=l^{2}s$ so that $\big(\frac{m}{d}\big)=\big(\frac{m}{s}\big)$. Further without loss of generality, we assume that $l$ and $s$ are monic. 
Then using $\sum_{l^{2}|d}\mu(d)=\mu^{2}(d)$, we have

\begin{align*}
S &=\sum_{\deg m=M}\sum_{\deg d=Z}\sum_{l^{2}|d}\mu(l)\Big(\frac{m}{s}\Big)\\
&=\sum_{\deg m=M}\sum_{\deg l \le \frac{Z}{2}}\mu(l)\sum_{\deg s=Z-2\deg l}\Big(\frac{m}{s}\Big)
\end{align*}
If $\deg l =Z/2$, then $s=1$. For such $l$, the corresponding contribution in $S$ is 
\[
\sum_{\deg m=M}\sum_{\deg l=\frac{Z}{2}}\mu(l).
\]
For $Z \ge 2$, the sum $\sum_{\deg l=\frac{Z}{2}}\mu(l)$ is zero by Lemma \ref{lemmasec42}. Since $Z \ge M+1 >2$, we deduce that the contribution in $S$ corresponding to $s=1$ is $0$. 
\vskip 5pt
\noindent
Therefore,
\begin{align*}
S &=\sum_{\deg m=M}\sum_{\deg l < \frac{Z}{2}}\mu(l)\sum_{\substack{\deg s=Z-2\deg l\\ s \neq 1}}\Big(\frac{m}{s}\Big)\\
&=\sum_{\deg l < \frac{Z}{2}}\mu(l)\sum_{\deg m=M}\sum_{\substack{\deg s=Z-2\deg l\\ s \neq 1}}\Big(\frac{m}{s}\Big),
\end{align*}
\noindent
which is
\begin{equation}\label{inequality}
\le \sum_{\deg l < \frac{Z}{2}}|\sum_{\deg m=M}\sum_{\substack{\deg s=Z-2\deg l\\ s \neq 1}}\Big(\frac{m}{s}\Big)|.
\end{equation}
\noindent
Observe that since $m$ satisfies equation (\ref{dioeq1}), and since we have assumed that $\deg f$ and $g$ are odd in (\ref{dioeq1}), $m$ cannot be a square in $\mathcal A$. 
Also $\deg m= M >1$ implies that $m \notin \mathbb{F}_{q}^{\times}$.
\vskip 5pt
\noindent
Thus appealing to the first part of lemma \ref{lemmasec43} we deduce that if $M \le Z-2\deg l$, then
\[
\sum_{\substack{\deg s=Z-2\deg l\\ s \notin \mathbb{F}_{q}^{\times}}}\Big(\frac{m}{s}\Big)=0,
\]
\noindent
while if $M \ge Z-2\deg l$, then from the second part of Lemma \ref{lemmasec43} we have
\[
\sum_{\deg m=M}\sum_{\substack{\deg s=Z-2\deg l\\ s \notin \mathbb{F}_{q}^{\times}}}\Big(\frac{m}{s}\Big) \le \Big(1-\frac{1}{q}\Big)q^{\frac{Z}{2}-\deg l +M}.
\]
Summing over $l$ in (\ref{inequality}) we deduce that $S \le q^{M+\frac{Z}{2}}$. Thus the contribution from $d \neq 1$ is less than
\[
q^{M+T}\sum_{Z \ge M+1}q^{-Z/2}= q^{M+T}q^{-\frac{M+1}{2}}\Big(1-\frac{1}{\sqrt{q}}\Big)^{-1}=O\big(q^{M/2+T}\big)
\]
This completes the proof of the lemma.
\end{proof}
\noindent
As an immediate consequence of Lemma \ref{average}, from (\ref{lb61}) we have
\[
N_{1} \asymp q^{M+N-T}+o(q^{M+\frac{L}{3}+\frac{2T}{3}}).
\]
\noindent
\textit{Estimation of} $N_{2}$: In order to estimate $N_{2}$, once again, we fix $m$ and $t$ and count the number of $n$ with $\deg n=N$ such that $\frac{n^{2}-m^{g}}{t^{2}}$ 
divisible by $p^{2}$ for some prime $p$ with $\log L<\deg(p) \le Q=\frac{L-T+2\log L}{3}$. Therefore the sum over $n$ that we seek is
\begin{equation}\label{lb8}
\sum_{\log L < \deg p \le Q}\sum_{\substack{\deg n=N \\ n^{2} \equiv m^{g} \pmod{p^{2}t^{2}}}}1.
\end{equation}
\noindent
Following the same line of argument as in the estimation of $N_{1}$ we deduce that the sum in (\ref{lb8}) is equal to 
\begin{equation}\label{lb9}
\sum_{\log L<\deg p\leq Q} \Big(\frac{q^{N}\rho_{m}(p^{2}t^{2})}{|p^{2}t^{2}|}+O\big(\rho_{m}(p^{2}t^{2}\big)\Big).
\end{equation}
\noindent
Since $\rho_{m}\big(p/(p,t)\big)\le 2$ the main term in (\ref{lb9}) is 
\begin{align*}
& q^{N-2T}\rho_{m}(t^{2}) \sum_{\log L < \deg p \le Q}\frac{\rho_{m}\big(p/(p,t)\big)}{|p|^{2}}\\
& \le q^{N-2T}\rho_{m}(t^{2})\sum_{\log L \le \deg p \le Q}\frac{2}{|p|^{2}} = 2q^{N-2T}\rho_{m}(t^{2})\sum_{Y=\log L}^{Q} \sum_{\deg p=Y}\frac{1}{|p|^{2}} \\
&= 2q^{N-2T}\rho_{m}(t^{2})\sum_{Y=\log L}^{Q} q^{-2Y}\sum_{\deg p=Y}1 = 2q^{N-2T}\rho_{m}(t^{2})\sum_{Y=\log L}^{Q} q^{-2Y}\pi(Y) \\
& \le 2q^{N-2T}\rho_{m}(t^{2})\sum_{Y=\log L}^{Q} q^{-2Y} q^{Y}/Y \quad \text{(by Lemma \ref{lemmasec41})} \quad \\
& \le \frac{2q^{N-2T}\rho_{m}(t^{2})}{\log L}\sum_{Y=\log L}^{Q} q^{-Y} \le \frac{2q^{N-2T}\rho_{m}(t^{2})}{q^{\log L}\log L}\Big(1-\frac{1}{q}\Big)^{-1} \\ 
&=\frac{2q^{N-2T}\rho_{m}(t^{2})}{L\log L}\Big(1-\frac{1}{q}\Big)^{-1} \ll \frac{q^{N-2T}\rho_{m}(t^{2})}{L}.
\end{align*}
\noindent
From
\[
\rho_{m}(p^{2}t^{2}\big)=\rho_{m}(t^{2})\rho_{m}\big(p^{2}/(p,t)^{2}\big)=\rho_{m}(t^{2})\rho_{m}\big(p/(p,t)\big) \le 2\rho_{m}(t^{2}),
\]
we deduce that the remainder term in (\ref{lb9}) is 
\begin{equation}\label{lb10}
O\big(\rho_{m}(t^{2})\sum_{\log L < \deg p\leq Q}1 \big).
\end{equation}
\noindent
Now by Lemma \ref{lemmasec41}
\[
\sum_{\log L < \deg p \le Q}1 \le \sum_{D=\log L}^{Q}\frac{q^{D}}{D}.
\]
Using Euler's summation formula it can be verified that
\[
\sum_{D=\log L}^{Q}\frac{q^{D}}{D} \ll q^{Q}/Q.
\]
\noindent
Now,
\[
\frac{q^{Q}}{Q}=\frac{q^{L/3}q^{-T/3}q^{2\log L/3}}{\frac{L}{3}-\frac{T}{3}+\frac{2\log L}{3}}=\frac{3 q^{L/3}q^{-T/3} L^{2/3}}{L(1-\frac{T}{L}+\frac{2\log L}{L})}.
\]
In the end we will take $T$ to be a constant ($<1$) multiple of $L$. Therefore, we conclude from above that
\[
\frac{q^{Q}}{Q} \ll q^{L/3}q^{-T/3} L^{-1/3}=o(q^{L/3}q^{-T/3}).
\]
Thus,
\[
\sum_{\log L < \deg p \le Q}1=o(q^{L/3}q^{-T/3}).
\]
Using this estimate in (\ref{lb10}) we deduce that the remainder term in (\ref{lb9}) is $o(q^{L/3}q^{-T/3}\rho_{m}(t^{2}))$.
\vskip 5pt
\noindent
Therefore the sum over $n$ in (\ref{lb8}) is
\begin{equation}\label{lb11}
\sum_{\log L < \deg p \le Q}\sum_{\substack{\deg n=N \\ n^{2} \equiv m^{g} \pmod{p^{2}t^{2}}}}1 \ll \frac{q^{N-2T}\rho_{m}(t^{2})}{L}+o(q^{L/3}q^{-T/3}\rho_{m}(t^{2})).
\end{equation}
\noindent
Summing over all monic $m$ and $t$ in (\ref{lb11}) with $\deg m=M$ and $\deg t=T$, and using Lemma \ref{average} we get
\[
N_{2} \ll \frac{q^{M+N-T}}{L}+o(q^{M+\frac{L}{3}+\frac{2T}{3}}).
\]
\noindent
\textit{Estimation of} $N_{3}$: If $(m,n,t)$ is a tuple counted in $N_{3}$, then 

\begin{equation}\label{lb12}
n^{2}-m^{g}=\beta p^{2}t^{2}, 
\end{equation}
\noindent
for some monic prime $p$ with $\deg p > Q$ and some $\beta \in \mathcal A$. Clearly, $\deg \beta < L-2Q=(L+2T-4\log L)/3$. As $m$, $n$ and $t$ are monic and pairwise relatively prime, 
for fixed $m$ and $\beta$ with $\deg m=M$, and $\deg \beta < L-2Q$, the number of monic $n$ and $t$ satisfying (\ref{lb12}) 
is bounded by the number of solutions to the equation 

\begin{equation}\label{lb13}
m^{g}=x^{2}-\beta y^{2} 
\end{equation}
with $x$ and $y$ monic and co-prime. Assuming that such $x$ and $y$ exists, the ideal $(m)^{g}$ factors in $\mathcal{A}[\sqrt{\beta}]$ as

\[
m^{g}=(x+y\sqrt{\beta})(x-y\sqrt{\beta}).
\]
Working similarly as in Proposition \ref{construction}, it can be seen that any common factor of the ideals $(x+y\sqrt{\beta})$ and $(x-y\sqrt{\beta})$ contains $m^{g}$ 
and $x$. But $(m^{g},x)=1$ as $x$ and $y$ are co-prime, hence any common factor of $(x+y\sqrt{\beta})$ and $(x-y\sqrt{\beta})$ must be the whole ring $\mathcal{A}[\sqrt{\beta}]$. 
Therefore the ideals $(x+y\sqrt{\beta})$ and $(x-y\sqrt{\beta})$ are co-prime. From unique factorization of ideals of $\mathcal{A}[\sqrt{\beta}]$ we have
\[
(x+y\sqrt{\beta})=\mathfrak{a}^{g} \quad \text{and} \quad (x-y\sqrt{\beta})=\bar{\mathfrak{a}}^{g},
\]
\noindent
for some ideal $\mathfrak{a}$ and its conjugate $\bar{\mathfrak{a}}$ in $\mathcal{A}[\sqrt{\beta}]$. Thus the number of solutions in $x$ and $y$ to (\ref{lb13}) is bounded 
by the number of factorizations of the ideal $(m)$ into the product $\mathfrak{a} \bar{\mathfrak{a}}$. It can be easily verified that the number of such factorizations of the 
ideal $(m)$ in $\mathcal{A}[\sqrt{\beta}]$ is $\le d(m)$. Thus for fixed $m$ and $\beta$, the number of choices for $n$ and $t$ satisfying (\ref{lb12}) is $\le d(m)$. 
From Proposition $2.5$ of \cite{R} it follows that $\sum_{\substack{m-\text{monic} \\ \deg m=M}}d(m)=q^{M}(M+1)$. Therefore $N_{3}$ is $\le$ (number of choices of 
$\beta$)($\sum_{\substack{m-\text{monic} \\ \deg m=M}}d(m)$) which is
\begin{align*}
& \le(1+q+q^{2} \cdots +q^{L-2Q}) \sum_{\substack{m-\text{monic} \\ \deg m=M}}d(m) \\
&= \frac{(q^{L-2Q+1}-1)}{q-1} q^{M}(M+1) \\
& \le q^{L-2Q+1}q^{M}(M+1) \\
&= q\cdot q^{(L+2T-4\log L)/3}q^{M}(M+1) \\
&=q^{L/3}q^{2T/3}q^{M}qL^{-4/3}(M+1).
\end{align*}
\noindent
Noting from (\ref{MNT}) that $M < L$, we conclude 
\[
N_{3} \le q^{L/3}q^{2T/3}q^{M}qL^{-4/3}(M+1) \le q^{L/3}q^{2T/3}q^{M}qL^{-1/3}=o(q^{M+\frac{L}{3}+\frac{2T}{3}}),
\]
as desired.

\section{Proof of Lemma \ref{upperbound}}\label{lemproof2}

Let $\mathcal{S}$ denote the set of monic tuples $(m_{1},n_{1},t_{1};m_{2},n_{2},t_{2})$ such that $\dfrac{n_{1}^{2}-m_{1}^{g}}{t_{1}^{2}}=\dfrac{n_{2}^{2}-m_{2}^{g}}{t_{2}^{2}}$ 
with $\deg m_{i}=M$, $\deg n_{i}=N$, $\deg t_{i}=T$; $(m_{i},n_{i})=(m_{i},t_{i})=1$, and $(m_{1},n_{1},t_{1}) \neq (m_{2},n_{2},t_{2})$. It can be seen that for a square-free 
$f$, if $(m_{1},n_{1},t_{1})$ and $(m_{2},n_{2},t_{2})$ are solutions to equation (\ref{dioeq1}) of Section \ref{count}, then $(m_{1},n_{1},t_{1};m_{2},n_{2},t_{2}) \in \mathcal{S}$.
 For a fixed square-free $f$, the number of such tuples is $\mathcal{R}(f)\big(\mathcal{R}(f)-1\big)$. Thus
\[
\sum_{\deg f=L}\mathcal{R}(f)\big(\mathcal{R}(f)-1\big) \le |\mathcal{S}|.
\]
\vskip 5pt
\noindent
For $(m_{1},n_{1},t_{1};m_{2},n_{2},t_{2}) \in \mathcal{S}$ we have
\[
t_{2}^{2}(n_{1}^{2}-m_{1}^{g})=t_{1}^{2}(n_{2}^{2}-m_{2}^{g}).
\]
Rearranging we have 
\[
(t_{1}n_{2}+t_{2}n_{1})(t_{1}n_{2}-t_{2}n_{1})=t_{1}^{2}m_{2}^{g}-t_{2}^{2}m_{1}^{g}.
\]
Since $\deg (t_{1}^{2}m_{2}^{g}-t_{2}^{2}m_{1}^{g})\le Mg+2T < 3L$, for fixed $m$ and $t$, the number of choices for $n_{1}$ and $n_{2}$ is bounded by 
$d(t_{1}^{2}m_{2}^{g}-t_{2}^{2}m_{1}^{g})$, provided $t_{1}^{2}m_{2}^{g} \neq t_{2}^{2}m_{1}^{g}$. However, if $t_{1}^{2}m_{2}^{g}=t_{2}^{2}m_{1}^{g}$, 
then from $(m_{i},t_{i})=1$ and since $g$ is odd, we have $t_{1}=t_{2}$, $m_{1}=m_{2}$, and consequently $n_{1}=n_{2}$, contradicting the fact that 
$(m_{1},n_{1},t_{1}) \neq (m_{2},n_{2},t_{2})$.
\vskip 5pt
\noindent
Now $d(t_{1}^{2}m_{2}^{g}-t_{2}^{2}m_{1}^{g})=O(q^{\epsilon L})$.
\vskip 5pt
\noindent
Thus summing over $m_{i}$ and $t_{i}$ for $i=1,2$ we have
\begin{align*}
\sum_{\deg f=L}\mathcal{R}(f)\big(\mathcal{R}(f)-1\big) &\le \sum_{\deg m_{i}=M}\sum_{\deg t_{i}=T}d(t_{1}^{2}m_{2}^{g}-t_{2}^{2}m_{1}^{g})\\
&\ll q^{\epsilon L}\sum_{\deg m_{i}=M}\sum_{\deg t_{i}=T}1\\
&=q^{\epsilon L+2M+2T}.
\end{align*}

\section{Proof of the Theorem \ref{thm1}}\label{thmproof}
In this section we first determine a suitable optimal value of the parameter $T$ so that the inequality (\ref{MNT1}) is justified.
\vskip 5pt
Substituting the values of $M$ and $N$ from (\ref{MNT}) in (\ref{MNT1}) and rearranging terms we obtain 

\begin{equation}\label{LT1}
T/L \ge \frac{(g-2)}{4(g+1)}-\frac{\epsilon g}{2(g+1)}.
\end{equation}

Thus in view of (\ref{LT1}), the obvious optimal choice for $T/L$ is
\[
T/L=\frac{g-2}{4(g+1)}.
\]
Therefore we take 
\begin{equation}\label{TLEQ}
T=\frac{L(g-2)}{4(g+1)}.
\end{equation}
\noindent
Now substituting the value of $T$ from (\ref{TLEQ}) in (\ref{nglt}), we conclude that the number of solutions to equation (\ref{dioeq1}) is
\[
\gg q^{L(\frac{1}{2}+\frac{3}{2(g+1)}-\epsilon)}.
\]
\noindent
Therefore, it follows from Proposition \ref{construction} that
\[
N_{g}(L) \gg q^{L(\frac{1}{2}+\frac{3}{2(g+1)}-\epsilon)},
\]
\noindent
and this completes the proof of the Theorem \ref{thm1}.

\end{document}